\documentclass[12pt,reqno]{amsart}
\usepackage{graphicx,subfig,mathrsfs,undertilde}
\input RCT.sty

\def\End{\mathrm{End}}

\title{Quantum Observables on a Completely Simple Semigroup}
\author{Ph. Feinsilver}
\address{Department of Mathematics\\
Southern Illinois University\\
Carbondale, IL\\
62901 USA}

\begin{document}

\maketitle
\begin{abstract}{ Completely simple semigroups arise as the support of limiting measures of random walks on semigroups. Such a  limiting
measure is supported on the kernel of the semigroup. Forming tensor powers of the random walk leads to a hierarchy of the limiting kernels.
Tensor squares lead to quantum observables on the kernel. Recall that zeons are bosons modulo the basis elements squaring to zero.
Using zeon powers leads naturally to quantum observables which reveal the structure of the kernel. Thus asymptotic information about
the random walk is related to algebraic properties of the zeon powers of the random walk.}
\end{abstract}

\thispagestyle{empty}

\section{Introduction}
This work is based on the connection between graphs and semigroups developed by Budzban and Mukherjea \cite{BuMu} in the context of their study
of the Road Coloring Problem [RCP]. 
We restrict throughout to finite transformation semigroups, that is, semigroups of functions acting on a finite set, $V$.
In their interpretation of the RCP, a principal question is whether one can determine the rank of the kernel of a semigroup
from its generators. Questions in general about the kernel of a semigroup are of interest. We will explain this terminology
as we proceed.  \bigskip

The study of probability measures on semigroups by Mukherjea \cite{M}, Mukherjea-H\"ognas \cite{HM} in particular
reveals the result that the convolution powers of a measure on a finite semigroup, $\SG$, converge (in the Ces\`aro sense)
to a measure on the kernel, $\K$, the minimal ideal of the semigroup. The kernel is always \textsl{completely simple} (see Appendix \ref{sec:sgrpskrnls}
for precise details). In this context, a completely simple semigroup is described as follows. 
It is a union of disjoint isomorphic groups, called ``local groups". The kernel $\K$ can be organized as a two-dimensional grid with 
rows labelled by \textsl{partitions} of the underlying set $V$ and columns labelled by \textsl{range classes}, a family of equinumerous
sets into whose elements the blocks of a given partition are mapped by the members of $\K$. The cardinality of a range class is the \textsl{rank}
of the kernel. The idempotent  $e$ determined by a pair $(\P,\R)$, with $\P$ a partition of $V$ and $\R$ a range class, 
acting as the identity of the local group, fixes each of the elements of the range class.
Each of the mutually isomorphic local groups is isomorphic to a permutation group acting on its range class. We think of
the kernel as composed of \textsl{cells} each labelled by a pair ($\P,\R)$ and containing the corresponding local group.\bigskip

First we recall the basic theory needed about probability measures on semigroups, continuing with the related context of walks on graphs.
Briefly, the adjacency matrix of the graph is rescaled to a stochastic matrix, $A$. A decomposition of the adjacency matrix into binary stochastic
matrices, $C_i$, i.e., matrices representing functions acting on the vertices of the graph, is called a \textsl{coloring}. The coloring functions generate
the semigroup $\SG$ and the question is to determine features of the kernel $\K\subset\SG$ from the generators $C_i$. Here we present
novel techniques by embedding the semigroup in a hierarchy of tensor powers. We will find particular second-order tensors to provide
the operators giving us ``quantum observables". The constructions restrict to trace-zero tensors using zeon Fock space instead of the full tensor space. \bigskip

For convenient reference, we put some notations here.

\subsection{Notations}

We treat vectors as ``row vectors", i.e., $1\times n$ matrices, with $\e_i$ as the corresponding standard
basis vectors. The column vector corresponding to $v$ is $\dg v.$.  For a matrix $M$,
$M^*$ denotes its transpose.
And $\diag{v}$ denotes the diagonal matrix with entries $v_i$ of the vector $v$. \bigskip

The vector having components all equal to 1 is denoted $u$. And the matrix 
$J=\dg u.u$ has entries all ones. We use the convention that the identity matrix, $I$, as well as
$u$ and $J$, denote matrices of the appropriate size according to the context. \bigskip

\begin{section}{Probability measures on finite semigroups}
Since we have a discrete finite set, a probability measure is given as a function on the elements. 
The semigroup algebra is the algebra generated by the elements $w\in\SG$. In general, we consider formal
sums $\sum f(w)\,w$. The function $\mu$ defines a probability measure if 
$$0\le\mu(w)\le1\,,\forall w\in\SG\,, \text{ and  } \sum \mu(w)=1\ .$$
The corresponding element of the semigroup algebra is thus $\sum \mu(w)\,w$. The product of elements in the semigroup algebra yields
the convolution of the coefficient functions. Thus, for the convolution of two measures 
$\mu_1$ and $\mu_2$ we have
\begin{align*}
\sum_{w\in\SG} \mu_1*\mu_2(w)\,w&=\bigl(\sum_{w\in\SG} \mu_1(w)\,w\bigr)\bigl(\sum_{w'\in\SG} \mu_2(w')\,w'\bigr)\\
&=\sum_{w,w'\in\SG} \mu_1(w)\mu_2(w')\,ww'
\end{align*}
Hence the convolution powers, $\mu^{(n)}$ of a single measure $\mu=\mu^{(1)}$ satisfy
$$\sum_{w\in\SG} \mu^{(n)}(w)\,w=\bigl(\sum_{w\in\SG} \mu^{(1)}(w)\,w\bigr)^n$$
in the semigroup algebra. \bigskip

\begin{subsection}{Invariant measures on the kernel}
We can consider the set of probability measures with support the kernel of a finite semigroup of matrices (see Appendix \ref{sec:sgrpskrnls} for more details).
 Given such a measure $\mu$, it is 
\textit{idempotent} if $\mu*\mu=\mu$, i.e., it is idempotent with respect to convolution. An idempotent measure on a finite group must be the uniform distribution
on the group, its Haar measure. In general, we have 
\begin{theorem}\cite[Th. 2.2]{HM} \label{thm:idem}
An idempotent measure $\mu$ on a finite semigroup $\SG$ is supported on a completely simple subsemigroup, $\K'$ of $\SG$.
With Rees product decomposition $\K'=\ms.X'.\times\ms.G'.\times\ms.Y'.$, $\mu$ is a direct product of the 
form $\alpha\times\omega\times\beta$, where $\alpha$ is a measure on $\ms.X',.$ $\omega$ is Haar measure on $\ms.G',.$ and
$\beta$ is a measure on $\ms.Y'.$.
\end{theorem}
In our context, we will have the measure $\mu$ supported on the kernel $\K$ of $\SG$.
We identify $\ms.X.$ with the partitions of the kernel and $\ms.Y.$ with the ranges. The local groups are mutually isomorphic finite
groups with the Haar measure assigning $1/|G|$ to each element of the local group $G$. Thus, if $k\in\K$ has Rees product decomposition
$(k_1,k_2,k_3)$ we have
$$\mu(k)=\alpha(k_1)\beta(k_3)/|G|$$
where $|G|$ is the common cardinality of the local groups, $\alpha(k_1)$ is the $\alpha$-measure of the partition of $k$, and $\beta(k_3)$ is the
$\beta$-measure of the range of $k$.
\end{subsection} \bigskip

\begin{remark} Given a graph with $n$ vertices, we identify $V$ with the numbers $\{1,2,\ldots,n\}$. 
Generally, given $n$, we are considering semigroups of transformations acting on the set $\{1,2,\ldots,n\}$.  \bigskip

For a function on $\{1,2,\ldots,n\}$, the notation $f=[x_1x_2\ldots x_n]$ means that $f(i)=x_i$, $i\le 1\le n$. And the 
rank of $f$ is the number of distinct values $\{x_i\}_{1\le i\le n}$.
\end{remark}

\begin{example} Here is an example with $n=6$. Take two functions $r=[451314]$, $b=[245631]$ and generate the semigroup $\SG$.
We find the elements of minimal rank, in this case it is 3.
The structure of the kernel is summarized in  a table with rows labelled by the partitions and columns labelled by the range classes. The entry is
the idempotent with the given partition and range. It is the identity for the local group of matrices with the given partition and range.
\begin{equation*}\label{eq:k6}
\begin{array}{c|c|c|c|c|}
&\{1, 3, 4\}&\{1, 4, 5\}&\{2, 3, 6\}&\{2, 5, 6\}\vstrut\cr
\hline
\{\{1, 2\}, \{3, 5\},\{4, 6\} \}& [1 1 3 4 3 4]& [1 1 5 4 5 4]&[2 2 3 6 3 6]&[2 2 5 6 5 6] \vstrut\cr
\hline
\{\{1, 6\},\{2, 4\},  \{3, 5\}\}& [1 4 3 4 3 1]&[1 4 5 4 5 1]&[6 2 3 2 3 6]& [6 2 5 2 5 6] \vstrut\cr
\hline
\end{array}
\end{equation*}
The kernel has 48 elements, the local groups being isomorphic to the symmetric group $S_3$.  \bigskip

Each cell consists of functions with the given partition, acting as permutations on the range class. They are given
by matrices whose nonzero columns are in the positions labelled by the range class. The entries in a nonzero column are in rows
labelled by the elements comprising the block of the partition mapping into the element labelling that column.
We label the partitions 
$$\P_1=\{\{1, 2\}, \{3, 5\},\{4, 6\} \}\,,\qquad \P_2=\{\{1, 6\},\{2, 4\},  \{3, 5\}\}$$
and the range classes
$$\R_1=\{1, 3, 4\}\,,\quad \R_2=\{1, 4, 5\}\,,\quad \R_3=\{2, 3, 6\}\,,\quad \R_4=\{2, 5, 6\}\ .$$
The local group with partition $\P_1$ and range $\R_3$, for example, are the idempotent $[2 2 3 6 3 6]$ noted in the above table and the functions
$$\{\,[6 6 3 2 3 2], [6 6 2 3 2 3], [3 3 6 2 6 2], [3 3 2 6 2 6], [2 2 6 3 6 3]\,\}$$
isomorphic to $S_3$ acting on the range class $\{2,3,6\}$. For the function $ [3 3 2 6 2 6]$, in the matrix semigroup we have the correspondence
$$  [3 3 2 6 2 6]\  \longleftrightarrow\  \begin{bmatrix}
0& 0& 1& 0& 0& 0\\0& 0& 1& 0& 0& 0\\ 0& 1& 0& 0& 0& 0\\0& 0& 0& 0& 0& 1\\0& 1& 0& 0& 0& 0\\0& 0& 0& 0& 0& 1
\end{bmatrix}\ .
$$
\bigskip

Using the methods developed below, one finds the measure on the partitions to be
$$\alpha=[1/3,2/3]$$
while that on the ranges is
$$\beta=[4/9,2/9,1/9,2/9]$$
as required by invariance.
\end{example} 
\end{section}

\begin{section}{Graphs, semigroups, and dynamical systems}
We start with a regular $d$-out directed graph on $n$ vertices with adjacency matrix $\A$. 
Number the vertices once and for all and identify vertex $i$ with $i$ and vice versa. \bigskip

Form the stochastic matrix $A=d^{-1}\,\A$.  We assume that $A$ is irreducible and aperiodic.
In other words, the graph is strongly connected and the limit
$$\lim_{m\to\infty}A^m=\Omega$$
exists and is a stochastic matrix with identical rows, the invariant distribution for the corresponding Markov chain, which we denote
by $\pi=[p_1,\ldots,p_n]$. The limiting matrix satisfies
$$\Omega=A\Omega=\Omega A=\Omega^2$$
so that the rows and columns are fixed by $A$, eigenvectors with eigenvalue 1. \bigskip

Decompose $A=\frac{1}{d} C_1+\cdots +\frac{1}{d} C_d$, into binary stochastic matrices, 
``colors", corresponding to the $d$ choices moving from a given vertex to another. Each coloring matrix $C_i$ is the matrix
of a function $f$ on the vertices as in the discussion in \S2. Let
$\SG=\SG((C_1,C_2,\ldots,C_d))$ be the semigroup generated by the matrices $C_i$, $1\le i\le d$. Then
we may write the decomposition of $A$ into colors in the form
$$A=\mu^{(1)}(1)C_1+\cdots +\mu^{(1)}(d) C_d=\sum_{w\in\SG} \mu^{(1)}(w)\,w$$
with $\mu^{(1)}$ a probability measure on $\SG$, thinking of the elements of $\SG$ as words $w$, strings from the alphabet generated
by $\{C_1,\ldots,C_d\}$. We have seen in the previous section that
\begin{equation}\label{eq:ces}
A^m=\sum_{w\in\SG}\mu^{(m)}(w)\,w
\end{equation}
where $\mu^{(m)}$ is the $m^{\rm th}$ convolution power of the measure $\mu^{(1)}$ on $\SG$. \bigskip

The main limit theorem is the following
\begin{theorem}\cite[Th. 2.7]{HM} Let the support of $\mu^{(1)}$ generate the semigroup $\SG$. Then the Ces\`aro limit
$$\lambda(w)=\lim_{M\to\infty} \frac{1}{M}\,\sum_{m=1}^M \mu^{(m)}(w)$$
exists for each $w\in\SG$. The measure $\lambda$ is concentrated on $\K$, the kernel of the semigroup $\SG$. It has the canonical decomposition
$$\lambda=\alpha \times \omega \times \beta$$
corresponding to the Rees product decomposition of $X\times G\times Y$ of $\K$.
\end{theorem}

\begin{remark}
See \cite{BF} for related material, including a proof of the above theorem.
\end{remark}

 \bigskip

We use the notation $\avg {\cdot}.$ to denote averaging with respect to $\lambda$ over random elements $K\in\K$. Forming
$M^{-1}\,\sum\limits_{m=1}^M$  on both sides of \eqref{eq:ces}, letting $M\to\infty$ we see immediately that
$$\Omega=\avg K.$$
the average element of the kernel. We wish to discover further properties of the kernel by considering the behavior
of tensor powers of the semigroup elements.  \bigskip

From the above discussion, we
have the following.
\begin{proposition} \label{prop:hom}
Let $\phi:\SG\to \SG'$ be a homomorphism of finite matrix semigroups. Let
$$A_\phi=\mu^{(1)}(1)\phi(C_1)+\cdots +\mu^{(1)}(d) \phi(C_d)=\sum_{w\in\SG} \mu^{(1)}(w)\,\phi(w)$$
then the Ces\`aro limit 
$$\lim_{M\to\infty} \frac{1}{M}\,\sum_{m=1}^M A_\phi^{\,m}$$
exists and equals
$$\Omega_\phi=\avg \phi(K).$$
the average over the kernel $\K$ with respect to the measure $\alpha\times\omega\times\beta$ as for $\SG$.
\end{proposition}
\begin{proof} Checking that
$$A_\phi^{\,m}=\sum_{w\in\SG} \mu^{(m)}(w)\,\phi(w)$$
the result follows immediately.
\end{proof}

$\Omega_\phi$ satisfies the relations
$$A_\phi\Omega_\phi=\Omega_\phi A_\phi=\Omega_\phi^2=\Omega_\phi\ .$$
\begin{proposition}
The rows of $\Omega_\phi$ span the set of left eigenvectors of $A_\phi$ for eigenvalue 1. The columns of $\Omega_\phi$ span
the set of right eigenvectors of $A_\phi$ for eigenvalue 1.
\end{proposition}
\begin{proof}
For a left invariant vector $v$, we have 
\begin{align*}
v=vA_\phi=vA_\phi^{\,m}=v\Omega_\phi &&\text{ (taking the Ces\`aro limit)}
\end{align*}
Thus, $v$ is a linear combination of the rows of $\Omega_\phi$. Similarly for right eigenvectors.
\end{proof}
Now, as $\Omega_\phi$ is an idempotent, we have $\rk \Omega_\phi=\tr \Omega_\phi$ so
\begin{corollary} The dimension of the eigenspace of right/left invariant vectors of $A_\phi$ equals $\tr \Omega_\phi$.
\end{corollary}

\begin{remark}
We will only consider mappings $\phi$ that preserve nonnegativity. 
$A_\phi$  in general will be \textit{substochastic}, i.e., it may have zero rows or rows that do not sum to one.
As noted in Appendix \ref{app:b}, the Abel limits exist for the $A_\phi$ in agreement with the Ces\`aro limits.
Abel limits are suitable for computations, e.g. with Maple or Mathematica.
\end{remark}
\end{section}

\section{Tensor hierarchy}\label{sec:tensorh}

Starting with a set $V$ of  $n$ elements, let
$F(V)=\{f\colon V\to V\}$. For a field of scalars, take the rationals, $\QQ$, as the most natural 
for our purposes and consider the vector space $\V=\QQ^n$ with $\End(\V)$ 
the space of $n\times n $ matrices acting as endomorphisms of $\V$. We have the mapping
$$ F(V)\stackrel{\Mt}{\longrightarrow} \End(\V) $$
taking $f\in F(V)$ to $\Mt(f)\in \End(\V)$ defined by

\begin{equation}\label{Xf}
(\Mt(f))_{ij}=\begin{cases} 1, & \text{if }f(i)=j\cr 0, & \text{otherwise} \end{cases}
\end{equation}

Denote by $F_\circ(V)$ the semigroup consisting of $F(V)$ with the operation of composition, 
where we compose maps to the right: for $i\in V$, $i(f_1f_2)=f_2(f_1(i))$.
The mapping $\Mt$ gives a representation of the semigroup $F_\circ(V)$ by endomorphisms of $\V$, i.e.,
$\Mt(f_1f_2)=\Mt(f_1)\Mt(f_2)$. \bigskip

For $W$ of the form $\Mt(f)$ corresponding to a function $f\in F(V)$, the resulting components of 
$\tpow{W}{k}$ are components of the map on tensors given by the $k^{\rm th}$ kronecker power of $W$.
At each level $k$, there is an induced map
$$ \End(\V) \to \End(\tpow{\V}{k})\ ,\qquad \Mt(f)\to \tpow{\Mt(f)}{k}$$
satisfying 

\begin{equation}
\tpow{(\Mt(f_1f_2))}{k}=\tpow{(\Mt(f_1)\Mt(f_2))}{k}=\tpow{\Mt(f_1)}{k}\tpow{\Mt(f_2)}{k}
\end{equation}

giving, for each $k$, a representation of the semigroup $F_\circ(V)$ as endomorphisms
of $\tpow{\V}{k}$. \bigskip

What is the function, $f_k$, corresponding to $\tpow{\Mt(f)}{k}$, i.e., such that $\Mt(f_k)=\tpow{\Mt(f)}{k}$ ?
For degree 1, we have from \eqref{Xf}
\begin{equation}
\e_i\Mt(f)=\e_{f(i)}
\end{equation}

And for the induced map at degree $k$, taking products in $\V^{\otimes k}$, for a multi-set $\rI$,

\begin{align*}
\e_{\rI}\tpow{\Mt(f)}{k} &=(\e_{i_1}\Mt(f)) \otimes (\e_{i_2}\Mt(f)) \otimes \cdots  \otimes (\e_{i_k}\Mt(f)) \cr
                                &= \e_{f(i_1)} \otimes \e_{f(i_2)} \otimes \cdots \otimes \e_{f(i_k)}
\end{align*}

We see that the degree $k$ maps are those induced on multi-subsets of $V$ mapping
$$\{i_1,\ldots,i_k\} \to \{f(i_1),\ldots,f(i_k)\}$$
with no further conditions, i.e., there is no restriction that the indices be distinct.
This thus gives the second quantization of $\Mt(f)$ corresponding to the induced map,
the second quantization of $f$, extending the domain of $f$ from $V$ to $\V^{\otimes}$, the space of all tensor powers of $\V$.\bigskip

Our main focus in this work is on the degree two component, where we have a natural correspondence
of $2$-tensors with matrices.

\subsection{The degree 2 component of $\V^{\otimes}$} \label{sec:two}

Working in degree 2, we denote indices $\rI=(i,j)$, as usual, instead of $(i_1,i_2)$.\bigskip

For given $n$, $X$, $X'$, etc., are vectors in $\tpow{\V}{2}\approx\QQ^{n^2}$.\bigskip

As a vector space, $\V$ is isomorphic to $\QQ^n$. Denote by ${\rm End}(\V)$ the space of
matrices acting on $\V$. \bigskip

{\bf Definition}\quad
The mapping  
$${\rm Mat}\colon\,\tpow{\V}{2}\to {\rm End}(\V)$$
is the linear isomorphism taking
the vector $X=(x_{ij})$ to the matrix $\tilde X$ with entries
$$\tilde X_{ij}=x_{ij}\ .$$

We will use the explicit notation ${\rm Mat}(X)$ as needed for clarity.\bigskip

Equip $\tpow{\V}{2}$ with the inner product 
$$\avg X,X'.=\tr \tilde X^*\tilde X'$$
the star denoting matrix transpose. \bigskip

Throughout, we use the convention wherein repeated Greek indices are automatically summed over.
So we write
$$\avg X,X'.=X_{\lambda\mu} X'_{\lambda\mu}\ .$$

\subsection{Basic Identities}

Multiplying $X$ with $u$, we note
\begin{equation}\label{teqtrace1}
X\dg u.=\tr \tilde X J=u\dg  X.
\end{equation}

We observe

\begin{proposition} {\it (Basic Relations)} We have\bigskip

1. $\Mat(X\tpow{A}{2}) = \symm{A}{\tilde X}$.\bigskip

2. $\Mat(\tpow{A}{2}\dg  X.) = \sym{A}{\tilde X}$.
\begin{proof} We check \#1 as \#2 is similar. The matrix $\tpow{A}{2}$ has components
$$\tpow{A}{2}_{ab,cd}=A_{ac}A_{bd}$$
And we have
\begin{align*}
(X\tpow{A}{2})_{cd}&=\sum_{a,b} x_{ab}A_{ac}A_{bd}\\ &=(A^*\tilde X A)_{cd}
\end{align*}
\end{proof}
\end{proposition}

We see immediately 

\begin{proposition}\label{tpropXX} For any $A$ and $X$,  \bigskip
\begin{align*}
  \tilde X = \symm{A}{\tilde X}  &\Leftrightarrow  X\tpow{A}{2} = X\\
  \tilde X = \sym{A}{\tilde X}  &\Leftrightarrow  \tpow{A}{2}\dg  X. = \dg  X.\\
\end{align*}
\end{proposition}

\subsection{Trace Identities}

Using equation (\ref{teqtrace1}), we will find some  identities for traces of these quantities.\bigskip

\begin{proposition}\label{tpropTR} We have \bigskip

1. $\displaystyle X\tpow{A}{2}\dg  u. = \tr (\tilde X AJA^*)$ .\bigskip

2. $\displaystyle u\tpow{A}{2}\dg  X. = \tr (\tilde X A^*JA)$ .\bigskip

3. If $A$ is stochastic, then \ $\displaystyle X\tpow{A}{2}\dg  u. = \tr (\tilde XJ)$.
\end{proposition}

\begin{proof} We have, using equation (\ref{teqtrace1}) and Basic Relation 1,
\begin{align*}
X\tpow{A}{2}\dg  u.  &= \tr \Mat(X\tpow{A}{2})J\\
 &= \tr (\symm{A}{\tilde X}J) 
 \end{align*}
and rearranging terms inside the trace yields \#1. Then \#3 follows since $A$ stochastic implies $AJ=J=JA^*$.
And \#2 follows similarly, using the second Basic Relation in the equation
$\displaystyle u\tpow{A}{2}\dg  X. = \tr \Mat(\tpow{A}{2}\dg  X.)J$. \par
\hfill\qedhere
\end{proof}

Using  equation (\ref{teqtrace1}) directly for $X$, we have \bigskip
\begin{align*}
X(I-\tpow{A}{2})\dg  u.  &=  \tr (\tilde X (J-AJA^*))\\
u(I-\tpow{A}{2})\dg  X.  &=  \tr (\tilde X (J-A^*JA))
\end{align*}
And the first line reduces to 
$$X(I-\tpow{A}{2})\dg  u.  = 0$$
for stochastic $A$.  \bigskip

\begin{subsection}{Convergence to tensor hierarchy}
\begin{proposition} \label{prop:tens}
For $\l\ge 1$, define
$$A_{\otimes\l}=\frac{1}{d} \tpow{C_1}{\l}+\cdots +\frac{1}{d} \tpow{C_d}{\l}\ .$$
Then the Ces\`aro limit of the powers $A_{\otimes\l}^{\,m}$ exists and equals 
$$\Omega_{\otimes\l}=\avg \tpow{K}{\l}.$$
the average taken over the kernel $\K$ of the semigroup $\SG$ generated by $\{C_1,\ldots,C_d\}$.
\end{proposition}

We imagine the family $\{\tpow{\K}{\l}\}_{\l\ge1}$, as a hierarchy of kernels corresponding to the family of semigroups
generated by tensor powers $\tpow{C_i}{\l}$ in each level $\l$. Since each element $k\in\K$ is a stochastic matrix, in fact
a binary stochastic matrix, as each corresponds to a function, we have $kJ=J$ at every level. We thus have a mechanism to 
move down the hierarchy by the relation
$$\tpow{k}{\l}(J\otimes \tpow{I}{(\l-1)})=J\otimes \tpow{k}{(\l-1)}$$
which is composed of identical blocks $\tpow{k}{(\l-1)}$. \bigskip

Note that if $\rk \Omega_{\otimes\l}=1$, the invariant distribution can be computed as an element of the nullspace
of $I-\tpow{A}{\l}$. Otherwise, $\Omega_{\otimes\l}$ can be found computationally as the Abel limit 
$$\lim_{s\uparrow 1}(1-s) (I-s\tpow{A}{\l})^{-1}\ .$$

\end{subsection}

\begin{section}{The principal observables: $M$ and $N$ operators}
\subsection{Graph-theoretic context}
For the remainder of this paper, we work mainly in the context where the stochastic matrix $A=\frac12 \A$, with $\A$ the adjacency matrix of 
a 2-out regular digraph, assumed strongly connected and aperiodic. In other words, $A$ is an ergodic transition matrix for the Markov chain
induced on the vertices, $V$, of the graph. $A$ decomposes into two coloring matrices, which we call from now on $R$ and $B$, for
``red" and ``blue", respectively. Thus
$$A=\tfrac12(R+B)$$
where $R$ and $B$ are binary $n\times n$ stochastic matrices.  \bigskip

\begin{example} For a working example, consider
$$A= \left[ \begin {array}{cccc} 0&0&1/2&1/2\\0&0&1/2&1/2\\ 1/2&0&0&1/2\\ 0&1/2&1/2&0
\end {array} \right]\ .$$
The powers of $A$ converge to the idempotent matrix
$$\Omega= \left[ \begin {array}{cccc} 1/6&1/6&1/3&1/3\\1/6&1/
6&1/3&1/3\\1/6&1/6&1/3&1/3\\1/6&1/6&1/3&1/3\end {array} \right]  \ .$$

As one possible decomposition, we may take
$$R=\left[ \begin {array}{cccc} 0&0&0&1\\0&0&1&0\\ 1&0&0&0\\ 0&1&0&0\end {array} \right]\,,\qquad
B=\left[ \begin {array}{cccc} 0&0&1&0\\0&0&0&1\\ 0&0&0&1\\ 0&0&1&0\end {array} \right]$$
In the notation used previously, we write 
$$ R=[4312]\qquad\text{and}\qquad B=[3443]\ .$$
\end{example}
\subsection{Level 2 of the tensor hierarchy}
Let 
$$A_{\otimes 2}=\tfrac12(\tpow{R}{2}+\tpow{B}{2})$$
Then fixed points $X$ of $A_{\otimes 2}$ are vectors corresponding to solutions to the matrix equations
$$\half\,(\sym{R}{\tilde X}+\sym{B}{\tilde X})=\tilde X$$
and
$$\half\,(\symm{R}{\tilde X}+\symm{B}{\tilde X})=\tilde X$$

\subsubsection{$M$ and $N$ operators}
In Section \ref{sec:two}, we defined the map from vectors to matrices $X\to \tilde X=\Mat(X)$. Now it is convenient
to have a reverse map from matrices to vectors, say 
$$Y \to Y_{\rm vec}$$
where here we use $Y$ to denote a typical matrix to distinguish our conventional use of $X$ as a vector. \bigskip

Denote  by $\N$ the space of solutions to $$\half\,(\sym{R}{Y}+\sym{B}{Y})=Y \ .$$ \par
The nonnegative solutions denote by $\N_+$, and the space of nonnegative solutions with trace zero by $\N_0$.  \bigskip

Similarly, $\M$ denotes the space of solutions to $$\half\,(\symm{R}{Y}+\symm{B}{Y})=Y$$
with $\M_+$ nonnegative solutions, and $\M_0$, nonnegative trace zero solutions. \bigskip

Starting from the relation
$$\Omega_{\otimes 2}=\avg\tpow{K}{2}.$$
We have solutions in the space $\M$ in the form
$$X\Omega_{\otimes 2}=\avg K^* \tilde X K.=\sumprime k^* \tilde X k$$
and solutions in $\N$ in the form
$$\Omega_{\otimes 2}X=\avg \sym{K}{ \tilde X}.=\sumprime \sym{k}{\tilde X}$$
where the primed summation denotes an averaged sum, here with respect to the measure $\lambda$ on $\K$.

\begin{subsubsection}{Diagonal of $N\in\N$}
For any $Y$, observe that
$$(\sym{R}{Y})_{ij}=R_{i\lambda} Y_{\lambda\mu}R_{j\mu}=Y_{iR\,jR}$$
since $R$ acts as a function on the indices.  Similarly for $B$.\bigskip

\begin{proposition}\label{prop:diagn}
Let $N\in\N_+$. Then the diagonal of $N$ is constant.
\end{proposition}
\begin{proof}
Assume that $\displaystyle N_{11}=\max_i N_{ii}$, is the largest diagonal entry. Start with 
$$2\,N_{11}=N_{1R\,1R}+N_{1B\,1B}$$
Since $N_{11}$ is maximal, both terms on the right-hand side agree with $N_{11}$. Generally,
if $N_{ii}=N_{11}$ then $N_{iR\, iR}=N_{iB\,iB}=N_{11}$. Proceeding inductively yields $N_{ii}=N_{11}$ for all $i$ since
the graph is strongly connected.
\end{proof}
\end{subsubsection}

\begin{subsubsection}{Diagonal of $M \in \M$}
Define the diagonal matrix $\omega=\diag{\pi}$, where $\pi$ is the invariant distribution for $A$.\bigskip
\begin{proposition}\label{prop:diagm}
Let $M\in\M_+$. Then the diagonal of $M$ is a scalar multiple of $\omega$.
\end{proposition}
\begin{proof}
Denote the diagonal of $M$ by $D$, i.e., $M_{ii}=D_i$. Start with
$$2\,M_{ij}=(\symm{R}{M}+\symm{B}{M})_{ij}=R_{\lambda i} M_{\lambda\mu}R_{\mu j}+B_{\lambda i} M_{\lambda\mu}B_{\mu j}$$
Let $\bar\delta_{ij}=1-\delta_{ij}$, i.e., one if $i\ne j$, zero otherwise. Take $i=j$ and split off the diagonal terms on the right to get
\begin{align*}
2\,D_i &=R_{\lambda i} D_{\lambda}R_{\lambda i}+B_{\lambda i} D_{\lambda}B_{\lambda i}
              +R_{\lambda i} M_{\lambda\mu}R_{\mu i}\bar\delta_{\lambda\mu}+B_{\lambda i} M_{\lambda\mu}B_{\mu i}\bar\delta_{\lambda\mu}\\
&=R_{\lambda i} D_{\lambda}+B_{\lambda i} D_{\lambda}+R_{\lambda i} M_{\lambda\mu}R_{\mu i}\bar\delta_{\lambda\mu}+B_{\lambda i} M_{\lambda\mu}B_{\mu i}\bar\delta_{\lambda\mu}\\
&=2A_{\lambda i} D_{\lambda}+R_{\lambda i} M_{\lambda\mu}R_{\mu i}\bar\delta_{\lambda\mu}+B_{\lambda i} M_{\lambda\mu}B_{\mu i}\bar\delta_{\lambda\mu}
\end{align*}
using the fact that the entries of $R$ and $B$ are $0$'s and $1$'s. Thus, 
$$D_i-(DA)_i=\half\,(R_{\lambda i} M_{\lambda\mu}R_{\mu i}\bar\delta_{\lambda\mu}+B_{\lambda i} M_{\lambda\mu}B_{\mu i}\bar\delta_{\lambda\mu})$$
Now sum both sides over $i$. Since $A$ is stochastic, $\sum (DA)_i=\sum D_i$ and the left-hand side vanishes. Since the right-hand side is a 
sum of nonnegative terms adding to zero, the right-hand side is identically zero. Thus,
$$D=DA$$
is a left fixed point of $A$. By irreducibility, it must be a multiple of $\pi$.
\end{proof}
\end{subsubsection}

\begin{subsection}{Specification of operators $N$ and $M$}
In the following, we will use $M$ and $N$ to refer to the specific operators defined by
$$M=\Mat(u\Omega_{\otimes 2})=\avg K^* J K. $$
and
$$N=\Mat(\Omega_{\otimes 2}I_{\rm vec})=\avg KK^* . $$

Another special operator in $\M$ is defined as
$$\tilde M=\Mat(\Omega_{\rm vec}\Omega_{\otimes 2})=\avg K^* \Omega K. $$

We can pass to trace-zero operators by subtracting off the diagonals defining
\begin{align}\label{eq:MNdef}
M_0&=M-\tau\omega\\
N_0&=J-N
\end{align}
For now, $\tau$ is defined by the above relation. We will see that the diagonal of $N$ is $\diag{u}=I$. The reason for
complementing with respect to $J$, i.e. subtracting from 1 everywhere will become clear after Proposition \ref{prop:eestar}.\bigskip

Note that at level 1, 
\begin{align}\label{eq:upi}
u\Omega&=u\ud\pi=n\pi=n\pi A\\
\Omega \ud&=\ud =A\ud \label{eq:upii}
\end{align}
So for level two, we define the fields 
\begin{align*}
\pi_{\otimes 2}=u\Omega_{\otimes 2}\\
u_{\otimes 2}^\dagger=\Omega_{\otimes 2}I_{\rm vec}
\end{align*}
with analogous expressions for each level $\l$. Indeed, at level two these satisfy
\begin{align*}
\pi_{\otimes 2}A_{\otimes 2}=\pi_{\otimes 2}\\
A_{\otimes 2}u_{\otimes 2}^\dagger=u_{\otimes 2}^\dagger
\end{align*}
extending relations \eqref{eq:upi} and \eqref{eq:upii}. The families $\{u_{\otimes \l}\}$ and $\{\pi_{\otimes \l}\}$ provide
extensions to scalar fields the basic all-ones vector $u$ and invariant distribution $\pi$ of level one.
\end{subsection}
\begin{subsection}{Basic level two relations}
In detailing the structure of the kernel, we need notations corresponding to the range classes.
Given an idempotent $e\in\K$, define $\utilde\rho(e)$ to be the \textsl{state vector} of the range of $e$,
namely, it is a $0$\,-$1$ vector with $\utilde\rho(e)_i=1$ exactly when $ie=e$. The corresponding matrix
$$\rho(e)=\diag{\utilde\rho(e)}$$
\begin{proposition}\label{prop:level2}
We have the relations:
\begin{align*}
\Mat(\Omega_{\rm vec}\Omega_{\otimes 2})&=\tilde M=\frac{n}{r^2}\,\avg\utilde\rho^\dagger \utilde\rho.
 \qquad &\qquad \Mat(\Omega_{\otimes 2}\Omega_{\rm vec})&= \Omega\Omega^*=(\textstyle\sum p_i^2)\,J\\
\Mat(I_{\rm vec}\Omega_{\otimes 2})&= \avg K^*K.=n\omega\qquad &\qquad \Mat(\Omega_{\otimes 2}I_{\rm vec})&=N\\
\Mat(J_{\rm vec}\Omega_{\otimes 2})&= M\qquad &\qquad \Mat(\Omega_{\otimes 2}J_{\rm vec})&= J\\
\end{align*}
{\rm Note we are writing $J_{\rm vec}$ for parallel formulation as an alternative for $u$. }\bigskip

\begin{proof} We will give the proof of the first relation later, as it relies on Friedman's Theorem. Consider
\begin{align*}
\Mat(\Omega_{\otimes 2}\Omega_{\rm vec})&=\avg \sym{K}{ \Omega}.=\avg \Omega K^*.\\
&=\Omega\Omega^*=\ud\pi\pi^\dagger u=(\pi\pi^\dagger)\ud u
\end{align*}
where we use the fact that elements of $\K$ are stochastic matrices. Next,
observe that $k^*k$ is always diagonal. Indeed
$$(k^*k)_{ij}=k_{\lambda i}k_{\lambda j}$$
indicates a summation over $\lambda$ such that $k$ maps $\lambda$ to both $i$ and $j$. In fact it is the size of the block of the
partition that $k$ maps into $i$ when $i=j$. Furthermore, for any vector $v$, we have the tautology
$$\diag{v}=\diag{\diag{v}\ud}$$
We have
$$\Mat(I_{\rm vec}\Omega_{\otimes 2})=\avg K^* K. $$
and we compute
$$\avg K^*K{\dg  u.}.=\avg K^*{\dg  u.}.=\Omega^*\dg  u.=n\,\dg  \pi.\ .$$
hence the result. And
$$\Mat(\Omega_{\otimes 2}I_{\rm vec})=\avg \sym{K}{}.=N$$
by our definition. We have
$$\Mat(J_{\rm vec}\Omega_{\otimes 2})=\avg K^* J K.$$
the definition of $M$. And finally,
$$\Mat(\Omega_{\otimes 2}J_{\rm vec})=\avg \sym{K}{J}.=J$$
immediately.
\end{proof}
\end{proposition}

For some trace calculations, note that since $\tr\omega=1$, equation \eqref{eq:MNdef} effectively says 
$$\tau=\tr M$$
Now, we note that
$$\tr NJ=\tr \avg KK^*J.=\tr\avg K^*JK.=\tr M=\tau$$
as well.
\begin{subsubsection}{Interpretation of $\tau$}
In describing the kernel, we have two basic constants, the rank $r$ and $\tau$. We want to interpret $\tau$ as describing
features of the kernel. Start with 
\begin{proposition} \label{prop:eestar}
 Consider an idempotent $e\in\K$, determined by a partition and range $(\P,\R)$.
For all $k$ in the local group $G_e$ for which $e$ acts as the identity,
$$kk^*=ee^*$$
is a $0$\,-$1$ matrix such $(kk^*)_{ij}=1$ if $i$ and $j$ are in the same block of $\P$.
\end{proposition}
\begin{proof} Writing out $k_{i\lambda}k_{j\lambda}$ shows that this sum counts 1 precisely when both $i$ and $j$ map
to the same range element, i.e., they are in the same block of $\P$.
\end{proof}

We have an interpretation of the entries $N_{ij}$ as the probability that vertices $i$ and $j$ appear in the same block of a partition.
Thus, the entry $(J-N)_{ij}$ gives the probability that $i$ and $j$ are \textsl{split}, i.e., they never appear together in the same block of a partition. \bigskip

\begin{remark} Note that $ee^*$ is the same for all cells in a given row of the kernel. And is independent of $k$ in a given cell. Thus,
$$N=\avg KK^*.=\avg EE^*.$$
where it is sufficient in this last average to consider only the idempotents in any single column of the kernel.
\end{remark}

The column sum $(uee^*)_i$ is the size of the block of $\P$ containing $i$. Now $\tr Jee^*=\tr Jkk^*$ sums all the ones in $kk^*$, and for all local $k$,
$$\tr k^*Jk=\tr Jkk^*= \text{ equals the sum of the squares of the blocksizes of }\P.$$
It follows that 
\begin{proposition} $\tau=\tr M$ is the average sum of squares of the blocksizes over the partitions $\P$ of $\K$.
\end{proposition}

We continue with some relations useful for deriving and describing properties of the kernel.
\end{subsubsection}
\end{subsection}
\begin{subsection}{Useful relations between an idempotent and its range matrix}
We make some useful observations regarding an idempotent $e$ and its range matrix $\rho(e)$. \bigskip

\begin{proposition} \label{prop:rhoe}
 For any idempotent $e$,  \bigskip

1. $\rho(e)e=\rho(e)\,,\qquad  e^*\,\rho(e)=\rho(e)\,,\qquad e\rho(e)=e$. \par
2. $ ee^*\rho(e)=e \vstrut$.
\begin{proof}
For \#1, the first and last relations follow since $e$ fixes its range. Transposing the first yields the second.
And \#2 follows directly.
\end{proof}
\end{proposition}

For $k\in\K$, let $\utilde\rho(k)$ denote the state vector of the range of $k$ and $\rho(k)$ the corresponding diagonal matrix.

\begin{corollary} \label{cor:erhoe}  \bigskip

1. For any $k\in\K$, $k\rho(k)=k$.\par
2. $e^*\utilde\rho(e)^\dagger=\dg \utilde\rho(e).\vstrut$.\par
3. $e\dg\utilde\rho(e).=\ud$.
\begin{proof}
For \#1, write $k=ke$ with $\rho(k)=\rho(e)$. Then 
$$k\rho(k)=ke\rho(e)=ke=k$$
as in the Proposition. And \#2 and \#3 follow from the Proposition via $\diag{v}\ud=\dg v.$.
\end{proof}
\end{corollary}

\end{subsection}
\end{section}

\begin{section}{Projections}\label{sec:proj}
We can compute the row and column projections in the kernel using our operators $M$ and $N$.  \bigskip

The idempotents having the same range $\R$, say, form a \textsl{left-zero semigroup}. That is, if $e_1$ and $e_2$ have the same
range, then $e_1e_2=e_1$, since $e_1$ determines the partition and the range is fixed by both idempotents. Correspondingly, 
the idempotents having the same partition $\P$, say, form a \textsl{right-zero semigroup}. That is, if $e_1$ and $e_2$ have the same
partition, then $e_1e_2=e_2$, since the partition for both is the same and an element in the range of $e_2$ is mapped by $e_1$ into an element
in the same block of the common partition, which in turn is mapped back to the original element by $e_2$.  \bigskip

\begin{remark} Note that in this terminology Proposition \ref{prop:rhoe} says that the pair $\{e,\rho(e)\}$ form a left-zero semigroup.
\end{remark}

\begin{notation}
It is convenient to fix an enumeration of the partitions $\{\P_1,\P_2,\ldots\}$ and ranges $\{\R_1,\R_2,\ldots\}$.
We set $\rho_j$ to be the matrix $\diag{\utilde\rho_j}$ where  $\utilde\rho_j$ is the state vector for $\R_j$.
\end{notation}

\begin{subsection}{Column projections} \label{sec:cproj}
Denote $e_{ij}$ the idempotents for range $\R_j$. Define the column projection to be the average
$$P_j=\alpha_\lambda e_{\lambda j} \ .$$
The measure $\alpha$ is the component of $\lambda$ on the partitions. 

\begin{proposition}\label{prop:col}\hfill\break

1. $P_j^2=P_j$ is an idempotent.\medskip

2. $P_j=N  \rho_j$.
\end{proposition}
\begin{proof}
For \#1, with $u$ a vector of all ones,
\begin{align*}
P_j^2&=\alpha_\lambda \alpha_\mu e_{\lambda j} e_{\mu j}\cr
&=\alpha_\lambda \alpha_\mu u_\mu e_{\lambda j}\cr
&=\alpha_\lambda e_{\lambda j}=P_j
\end{align*}
using the fact that the idempotents with a common range form a left-zero semigroup and that
the $\alpha$'s are a probability distribution.\medskip

For \#2, recall the relation $e^* \rho(e)= \rho(e)$ from Proposition \ref{prop:rhoe}, and we have
$$N \rho_j=\alpha_\lambda e_{\lambda j}e_{\lambda j}^* \rho_j
=\alpha_\lambda e_{\lambda j} \rho_j=P_j$$
with $e_{i j} \rho_j=e_{ij}$ as each $e_{ij}$ fixes the common range.
\end{proof}
\end{subsection}

\begin{subsection}{Row projections}\label{sec:rproj}
Analogously, define $Q_i=\beta_\mu e_{i\mu}$, the average idempotent with common partition
$\P_i$. As in \#1 of Prop. \ref{prop:col}, $Q_i^2=Q_i$, where now the idempotents form a 
right-zero semigroup.  \bigskip

We need the fact [to be proved later] that the average of the range matrices is $r$ times $\omega=\diag{\pi}$. We state this in the form
\begin{equation}\label{eq:omega}
\omega=r^{-1}\beta_\mu \rho_\mu
\end{equation}
where $\beta$ is the part of the measure $\lambda$ on the range classes.  \bigskip

\begin{proposition}\label{prop:row}
$$Q_i=r\,e_ie_i^*\omega$$
where $e_i$ is any idempotent with partition $\P_i$.
\end{proposition}
\begin{proof}
First, recall that $e_{ij}e_{ij}^*=e_ie_i^*$ for any idempotent $e_{ij}$ with partition $\P_i$.
And $e_ie_i^* \rho(e_i)=e_i$. Using equation \eqref{eq:omega} we have
\begin{align*}
r\,e_ie_i^*\omega&=e_ie_i^*\beta_\mu \rho_\mu\cr
&=\beta_\mu e_{i\mu} e_{i\mu}^* \rho_\mu\cr
&=\beta_\mu e_{i\mu}=Q_i
\end{align*}
as required.
\end{proof}

\end{subsection}

\begin{subsection}{Average idempotent}
We now compute the average idempotent $\avg E.$.

\begin{theorem} The average idempotent satisfies
$$\avg E.=rN\,\omega$$
\begin{proof}
From Proposition \ref{prop:row}, we need to average $Q_i$ over the partitions. We get
$$\avg E.=\alpha_\lambda Q_\lambda
=r\,\alpha_\lambda e_{\lambda 1}e_{\lambda 1}^*\omega=r\,N\omega$$
using the idempotents with common range $\R_1$.
\end{proof}
\end{theorem}
\end{subsection}
\end{section}

\section{Equipartitioning}

In the current framework, as in \cite{BuMu}, we express Friedman's Theorem this way:
\begin{theorem}[\cite{BuMu,FRC}, version 2]\label{thm:bmf2}
For $k\in\K$, let $\utilde\rho(k)$ denote the $0$-$1$ vector with support the range of $k$. Then
$$\pi k=\frac{1}{r}\utilde\rho(k)$$
where $r$ is the rank of the kernel.
\end{theorem}

We will refer to this theorem as BM/F. \bigskip

\begin{corollary} For any $k\in\K$, we have
$$\pi kk^*=\frac{1}{r}\,u$$
\begin{proof} By Corollary \ref{cor:erhoe}, we have
$$k\rho(k)=k\qquad \text{which implies}\qquad k\utilde\rho(k)^\dagger=\ud$$
Now transpose.
\end{proof}
\end{corollary}
This corollary can be rephrased as stating that each block of any partition has $\pi$-probability equal to $1/r$.
Here we will give an independent proof with a functional analytic flavor. \bigskip

\begin{subsection}{A and $\Delta$}
Working with two colors $R$ and $B$, it is convenient to introduce the operator $\Delta$ 
$$\Delta=\half(R-B)$$
Then we check that:
$$A_{\otimes 2}=\tfrac12(\tpow{R}{2}+\tpow{B}{2})=\tpow{A}{2}+\tpow{\Delta}{2}$$
Hence, elements $Y\in\N$ satisfy
$$AYA^*+\Delta Y\Delta^*=Y$$
and those in $\M$ satisfy
$$A^*YA+\Delta^* Y\Delta=Y$$
In particular, $N$ and $M$ satisfy these equations accordingly:
\begin{align} \label{eq:N}
ANA^*+\Delta N\Delta^*=N \\
A^*MA+\Delta^* M\Delta=M  \label{eq:M}
\end{align}
\end{subsection}
\subsection{Friedman's Theorem d'apr\`es Budzban-Mukherjea}
First, some notation \bigskip

\begin{notation} Let $\E(\K)$ denote the set of all idempotents of the kernel $\K$.     \end{notation}

Start with

\begin{proposition} \label{propN}
For every $k\in\K$, $\pi\Delta k=0$. Equivalently, \hfill\break $\pi Rk=\pi Bk=\pi k$.
\label{proppidelta}
\end{proposition}

\begin{proof} Multiplying eq.\tsp(\ref{eq:N}) on the left by $\pi$ and on the right by $\pi^\dagger$, we have,
using $\pi A=\pi,\,A^*\pi^\dagger=\pi^\dagger$,
$$\pi{N}\pi^\dagger+\pi\sym{\Delta}{N}\pi^\dagger=\pi N\dg\pi.$$
Since $N=\avg KK^*.$, we have
\begin{align*}
\pi\sym{\Delta}{N}\pi^\dagger&=0\cr 
&=\avg \pi\Delta KK^*\Delta^*\pi^\dagger.=\avg\|\pi\Delta K\|^2. 
\end{align*}
All terms are nonnegative so that $\pi\Delta k=0$ for every $k$. From the definition of $\Delta$ the alternative
formulation results after averaging $\pi Rk=\pi Bk$ to $\pi Ak=\pi k$.\hfill\qedhere
\end{proof}

And directly from the definition of $N$
\begin{corollary} \label{cor:pideltan} We have $\pi\Delta N =0$.\end{corollary}

We extend to words in the semigroup generated by $R$ and $B$. \bigskip

\begin{proposition} For every $w\in\SG$, and $k\in\K$, $\pi wk=\pi k$.
\begin{proof} This follows by induction on $l$, the length of $w$. Proposition \ref{proppidelta}
is the result for $l=1$. Let $w_{l+1}$ be a word of
length $l+1$, then $w_{l+1}=$ one of $Rw_l$, $Bw_l$, where $w_l$ has length $l$. Since $w_lk\in\K$ for $k\in\K$, we have, by
Proposition \ref{proppidelta},
\begin{align*}
\pi w_{l+1} k &=\pi R w_lk = \pi B w_lk \cr &= \pi w_l k \hbox to 100 pt{}{\rm (averaging)}\cr
&= \pi k \hbox to 100pt{} \hbox{\rm (by induction)}
\end{align*}\hfill\qedhere
\end{proof} \label{proppwe}
\end{proposition}

\begin{proposition} \label{propBMF}
For any $e\in\E(\K)$, $w\in \SG$, $$\pi we=(1/r)\,\utilde\rho(e) \qquad \hbox{\ and\ }\qquad\pi w ee^*=(1/r)\,u$$
\begin{proof} From Corollary \ref{cor:pideltan}, $\pi\Delta N=0$.
Equivalently, $N\Delta^*\dg\pi.   = 0$. 
Hence, multiplying equation \eqref{eq:N} on the right by  $\dg\pi.$ , 
using $A^*\dg\pi. = \dg\pi.$, we find
$$                               AN\dg\pi.  = N \dg\pi.$$
So $N\dg\pi.$  is fixed by $A$, hence is a constant vector. Transposing back, say
$\pi N = au$ for some constant $a$. Now let $e\in \E(\K)$ and take the $e_i$ with the
same range as $e$ in the averaging defining $N$ as $\avg EE^*.$. Then from Proposition \ref{proppwe} we have
$\pi e e_i=\pi e_i$. Multiplying by $e_i^*$ and averaging yields
$ \pi eN=\pi N$, hence $ \pi eNe^*=\pi N e^*$.
Since the idempotents of a column are a left-zero semigroup, 
$ee_i=e,\,e_i^*e^*=e^*$. Thus, 
$$eNe^*=e\,\alpha_\mu e_\mu e_\mu^*\,e^*=ee^*$$
Thus on the one hand $ \pi eNe^*=\pi N e^*$
and on the other $\pi eNe^*=\pi ee^*$.
Or $\pi ee^* =\pi Ne^*=aue^*=au$ for any idempotent $e$.  Thus, by Prop. \ref{prop:rhoe}, \#2, $\pi e=a\utilde\rho(e)$ and
$\pi e\dg u.=1=ar$ yields $a=1/r$, where $r$ is the rank of the kernel. So, for any $w\in\SG$,
$\pi we=\pi e = (1/r)\utilde\rho(e)$ and Cor. \ref{cor:erhoe}, \#3 yields $\pi wee^*=(1/r)u$, as required.
\end{proof}
\end{proposition}

Now taking $w\in\K$ yields $\pi k=(1/r)\utilde\rho(k)$, the theorem of Budzban-Mukherjea/Friedman. \bigskip

Averaging the second relation in the above Proposition yields
\begin{corollary} \label{cor:pideltann} For any $w\in\SG$,
$\displaystyle \pi w N =\tfrac{1}{r}\,u$. \end{corollary}

\begin{subsubsection}{More on projections}
Referring to the column projections $P_j$ of \S\ref{sec:proj}, we have
$$\pi P_j=\pi N\rho_j=\tfrac{1}{r}\,u \rho_j=\tfrac{1}{r}\,\utilde\rho_j$$
And for row projections $Q_j$, via Proposition \ref{propBMF},
$$\pi Q_j=r \pi e_ie_i^*\omega =u\omega=\pi$$
Since $Q_j$ is a stochastic matrix, $Q_j\ud=\ud$, we have the result
\begin{proposition} Each of the row projections $Q_j$ commutes with $\Omega$.
\end{proposition}

\end{subsubsection}

\subsubsection{Local groups}
Let $G_{ij}$ be the local group with partition $\P_i$ and range $\R_j$.
For convenience, when averaging over $G_{ij}$, we use the abbreviated form
 $\displaystyle\sideset{}{'}\sum=\frac{1}{|G_{ij}|}\,\sum$.\bigskip

An important fact is
\begin{lemma}\label{lem:avg}
 The average group element is $\ud\utilde\rho/r$, specifically,
$$\sideset{}{'}\sum_{\substack{{}\\G_{ij}}} k = r^{-1}\ud\utilde\rho_j$$
\end{lemma}
\begin{proof}
From the structure of the kernel, for a fixed $e$, $e\K e$ is the local group having $e$ as local identity.
So take $e=e_{ij}$, the identity for $G_{ij}$.
Since $\Omega=\avg K.$, we have, via Prop.\tsp\ref{propBMF},
$$\avg eKe.=e\Omega e=e\ud\pi e=\ud\pi e=r^{-1}\ud\utilde\rho(e)$$\hfill\qedhere
\end{proof}

Now we see that to average over the kernel, we need only average the above relation over the ranges. This yields
\begin{equation}\label{eq:omm}
\Omega=r^{-1}\ud\beta_\mu\utilde\rho_\mu=r^{-1}\ud u\beta_\mu\rho_\mu=r^{-1}J\beta_\mu\rho_\mu
\end{equation}
since $\utilde\rho(e)=u\rho(e)$. 
Comparing diagonals, we thus derive equation \eqref{eq:omega}
$$\omega=r^{-1}\beta_\mu \rho_\mu $$
used in computing the row projections, finding as well
$$\pi=r^{-1}\beta_\mu\utilde\rho_\mu $$
We are now in a position to prove the first relation of Proposition \ref{prop:level2}:
\begin{theorem}
$$\Mat(\Omega_{\rm vec}\Omega_{\otimes 2})=\frac{n}{r^2}\,\avg\utilde\rho^\dagger \utilde\rho.=\frac{n}{r^2}\avg \rho J\rho .$$
\begin{proof}
\begin{align*}
\avg K^*\Omega K.&=\avg K^*\ud\pi K.=r^{-1}\avg K^*\ud\utilde\rho(K).\\
&=r^{-1}\sumprime_{i,j}\sumprime_{G_{ij}} k^*\ud\utilde\rho_j
\end{align*}
summing over the cells of the kernel, the primes indicating averaging. By Lemma \ref{lem:avg}, we have
\begin{align*}
r^{-1}\sumprime_{i,j}\sumprime_{G_{ij}}k^*\ud\utilde\rho_j&=r^{-2}\sumprime_{i,j}\dg \utilde\rho_j.u\ud \utilde\rho_j
=nr^{-2}\sumprime_{i,j}\dg \utilde\rho_j.\utilde\rho_j\\
&=nr^{-2}\sumprime_{i,j}\rho_j\ud u\rho_j=nr^{-2}\avg \rho J\rho.
\end{align*}\hfill\qedhere
\end{proof}
\end{theorem}

\section{Properties of $M$, $N$, and $\Omega$}
Here we look at some relations among the main operators of interest, including $\tilde M$. \bigskip

Start with
\begin{proposition}\hfill\break
1. $\Omega N=r^{-1}J\vstrut$.\medskip

2. $M\Omega=n\Omega^*\Omega$.\medskip

\begin{proof} $\Omega N=\ud\pi N=r^{-1}\ud u=r^{-1}J$ as required. And, noting that $J\Omega=n\Omega$,
$$M\Omega=\avg K^*JK\Omega.=\avg K^*J\Omega.=n\avg K^*.\Omega=\Omega^*\Omega$$
via $k\Omega=\Omega$, for all $k\in\K$.
\end{proof}
\end{proposition}

And consequently,
\begin{corollary} $NM$ commutes with $\Omega$.
\begin{proof}
First,
$$NM\Omega=n N\Omega^*\Omega=n(\Omega N)^*\Omega=(n/r)J\Omega=(n^2/r)\Omega$$
And second,
\begin{align*}
\Omega NM&=r^{-1}JM=r^{-1}\avg JK^*JK.=r^{-1}\avg (KJ)^*JK.\\
&=r^{-1}\avg J^2K.=r^{-1}nJ\Omega=(n^2/r)\Omega
\end{align*}
as required.
\end{proof}
\end{corollary}

Observe in the proof the relation
$$JM=n^2\Omega$$

The case of doubly stochastic $A$ turns out to be particularly interesting. \par
Recall the basic fact that two symmetric matrices commute if and only if their product is symmetric.

\begin{theorem} If $A$ is doubly stochastic, then $\{M, N, J\}$ generate a commutative algebra.
\begin{proof} If $A$ is doubly stochastic, then $u=n\pi$ and $n\Omega=J$. Thus 
$JM=n^2\Omega=nJ$ is symmetric. And $JN=n\Omega N=(n/r)J$ is symmetric. So $J$ commutes with
$M$ and $N$. We check that $M$ and $N$ commute:
\begin{align*}
MN &= \avg K^*{\dg  u.}uK.N\\ &= n^2\,\avg K^*{\dg  \pi.}\pi KN.\\
&= (n^2/r)\,\avg K^*{\dg  \pi.}u.\\ &= (n^2/r)\,\avg \Omega K.^*\\ &= (n^2/r)\,(\Omega^*)^2 =(n/r)\,J\,.\qedhere
\end{align*}
which is symmetric as well.
\end{proof}
\end{theorem}

Another feature involves $\tilde M$.
\begin{theorem}\label{thm:nom}
$$N\tilde M=\frac{n}{r}\,\Omega$$
\end{theorem}
\begin{proof}
We have
$$N \tilde M=nr^{-2} N \avg \rho J\rho.$$
Note that for any $P_j$, $P_j J=J$. Fixing a range, we drop subscripts for convenience. And
with $P=P_j$ the corresponding column projection,
$$N\rho J\rho=P J\rho=J\rho$$
By equation \eqref{eq:omm}, averaging yields $r\Omega$. Hence the result.
\end{proof}

\section{Zeon hierarchy}
Now we pass from the tensor hierarchy to the zeon hierarchy by looking at representations of the semigroup
$F_\circ(V)$ acting on zeon Fock space. We begin with the basic definitions and constructions. 
Then we proceed with statements parallelling those of \S\ref{sec:tensorh}, including proofs when they
illustrate important differences between the two systems. \bigskip

\begin{remark} The introductory material in this section is largely taken from \cite{PF}.
\end{remark}

Consider the exterior algebra generated by a chosen basis $\{\e_i\}\subset V$, with relations
$\e_i\wedge \e_j=-\e_j\wedge \e_i$. 
Denoting multi-indices by roman capital letters $\rI=(i_1,i_2,\ldots,i_k)$, $\rJ$, $\rK$, etc.,
at level $k$, a basis for $V^{\wedge k}$ is given by 
$$\e_{\rI}=\e_{i_1}\wedge\cdots \wedge \e_{i_k}$$
with $\rI$ running through all $k$-subsets
of $\{1, 2,\ldots,n\}$, i.e., $k$-tuples with distinct components. For $\Mt(f)$, $f\in F(V)$, define the matrix 
$$(\spow{\Mt(f)}{k})_{\rI\rJ}=|(\Mt(f)^{\wedge k})_{\rI\rJ}|$$
taking absolute values entry-wise.
It is important to observe that we are not taking the fully symmetric representation of $\End(\V)$, 
which would come by looking at the action on boson Fock space, spanned by
symmetric tensors. However, note that the fully symmetric representation is given by maps induced by the
action of $\Mt(f)$ on the algebra generated by commuting variables $\{\e_i\}$.
We take this viewpoint as the starting point of the construction of the zeon Fock space, $\Z$, to be
defined presently. \bigskip

\begin{definition}     
A {\sl zeon algebra\/} is a commutative, associative algebra
generated by elements $\e_i$ such that $\e_i^2=0$, $i\ge 1$.
\end{definition}

For a standard zeon algebra, $\Z$, the elements $\e_i$ are finite in number, $n$,
and are the basis of an $n$-dimensional vector space, $\V\approx \QQ^n$. 
We assume no further relations among the generators $\e_i$.
Then the $k^{\rm th}$ {\sl zeon tensor power\/} of $\V$, denoted $\spow{\V}{k}$,
is the degree $k$ component of the graded algebra $\Z$,  with basis 
$$ \e_{\rI}=\e_{i_1}\cdots \e_{i_k}$$
analogously to the exterior power except now the variables commute. The assumptions on the $\e_i$ imply that
$\spow{\V}{k}$ is isomorphic to the subspace of symmetric tensors 
spanned by elementary tensors with no repeated factors. {\it As vector spaces},
$$ \spow{\V}{k} \approx \V^{\wedge k} $$

The {\sl zeon Fock space\/} is $\Z$ presented as a graded algebra
$$ {\Z}=\QQ\oplus (\bigoplus_{k\ge1}\spow{\V}{k})$$
Since $\V$ is finite-dimensional, $k$ runs from 1 to $n=\dim \V$.\bigskip

A linear operator $W \in\End(\V)$ extends to the operator $\spow{W}{k}\in \End(\spow{\V}{k})$.
The {\sl second quantization\/} of $W$ is the induced map on $\Z$. \bigskip

For the exterior algebra, the $\rI\rJ^{\it th}$ component of $W^{\wedge k}$ 
is the determinant of the corresponding submatrix of $W$, with
rows indexed by $\rI$ and columns by $\rJ$. Having dropped the signs,
the $\rI\rJ^{\it th}$ component of $\spow{W}{k}$ is the permanent of the corresponding submatrix of $W$.\bigskip

For $W$ of the form $\Mt(f)$ corresponding to a function $f\in F(V)$, the resulting components of 
$\spow{W}{k}$ are exactly the absolute values of the entries of $W^{\wedge k}$, as we wanted.
At each level $k$, there is an induced map
$$ \End(\V) \to \End(\spow{\V}{k})\ ,\qquad \Mt(f)\to \spow{\Mt(f)}{k}$$
satisfying 

\begin{equation}\label{zhom}
\spow{(\Mt(f_1f_2))}{k}=\spow{(\Mt(f_1)\Mt(f_2))}{k}=\spow{\Mt(f_1)}{k}\spow{\Mt(f_2)}{k}
\end{equation}

giving, for each $k$, a representation of the semigroup $F_\circ(V)$ as endomorphisms
of $\spow{\V}{k}$. However, for general $W_1$, $W_2$, the homomorphism property, \eqref{zhom},
no longer holds, i.e., $ \spow{(W_1W_2)}{k}$ does not necessarily equal $\spow{W_1}{k}\,\spow{W_2}{k}$.
It is not hard to see that a sufficient condition is that
$W_1$ have at most one non-zero entry per column or that $W_2$ have at most one non-zero entry per row. 
For example, if one of them is diagonal,  as well as  the case where both correspond to functions.\bigskip

What is the function, $f_k$, corresponding to $\spow{\Mt(f)}{k}$, i.e., such that $M(f_k)=\spow{\Mt(f)}{k}$ ?
For degree 1, we have from \eqref{Xf}
\begin{equation}
\e_i\Mt(f)=\e_{f(i)}
\end{equation}

And for the induced map at degree $k$, taking products in $\Z$,

\begin{align*}
\e_{\rI}\spow{\Mt(f)}{k} &=(\e_{i_1}\Mt(f))\,(\e_{i_2}\Mt(f))\cdots (\e_{i_k}\Mt(f)) \cr
                                &= \e_{f(i_1)}\,\e_{f(i_2)}\cdots \e_{f(i_k)}
\end{align*}

We see that the degree $k$ maps are those induced on $k$-subsets of $V$ mapping
$$\{i_1,\ldots,i_k\} \to \{f(i_1),\ldots,f(i_k)\}$$
with the property that the image in the zeon algebra is zero if $f(i_l)=f(i_m)$ for any pair $i_l,i_m$.
Thus the second quantization of $\Mt(f)$ corresponds to the induced map,
the second quantization of $f$, extending the domain of $f$ from $V$ to
the power set $2^V$.\bigskip

Now we proceed to the degree 2 component in the zeon case.

\subsection{The degree 2 component of $\Z$}

Working in degree 2, we denote indices $\rI=(i,j)$, as usual, instead of $(i_1,i_2)$.\bigskip

For given $n$, $X$, $X'$, etc., are vectors in $\spow{\V}{2}\approx\QQ^{\binom{n}{2}}$.\bigskip

As a vector space, $\V$ is isomorphic to $\QQ^n$. 

Denote by ${\rm Sym}(\V)$ the space of
symmetric matrices acting on $\V$.\bigskip

{\bf Definition}\quad
The mapping  
$${\rm Mat}\colon\,\spow{\V}{2}\to {\rm Sym}(\V)$$
is the linear embedding taking
the vector $X=(x_{ij})$ to the symmetric matrix $\hat X $ with components 
$$\hat X_{ij}=
\begin{cases}
x_{ij}, & \text{for } i<j \cr
0, & \text{for } i=j
\end{cases}
$$
and the property $\hat X_{ji}=\hat X_{ij}$ fills out the matrix.\bigskip

We will use the explicit notation ${\rm Mat}(X)$ as needed for clarity.\bigskip

Equip $\spow{\V}{2}$ with the inner product 
$$\avg X,X'.=\frac12\,\tr \hat X\hat X'=X_{\lambda\mu} X'_{\lambda\mu}$$

\subsection{Basic Identities}

Multiplying $X$ with $u$, we observe that
\begin{equation}\label{eqtrace1}
X\dg u.=\frac12\,\tr \hat X J=u\dg  X.
\end{equation}
Observe also that if $D$ is diagonal, then $\tr D=\tr DJ$.\bigskip

\begin{proposition}\label{propBR}  {\it (Basic Relations)} We have\bigskip

1. $\Mat(X\spow{A}{2}) = \symm{A}{\hat X}-D^+$, where $D^+$ is a diagonal matrix satisfying $\vstrut\tr D^+=\tr \symm{A}{\hat X}$.\bigskip

2. $\Mat(\spow{A}{2}\dg  X.) = \sym{A}{\hat X}-D^-$, where $D^-$ is a diagonal matrix satisfying $\vstrut\tr D^-=\tr \sym{A}{\hat X}$.\bigskip

3. If $A$ and $X$ have nonnegative entries, then $D^+$ and $D^-$ have nonnegative entries. In particular, in that case,
vanishing trace for $D^{\pm}$ implies vanishing of the corresponding matrix.
\end{proposition}

\begin{proof} The components of $X\spow{A}{2}$ are
\begin{align*}
(X\spow{A}{2})_{\,ij}  &= \theta_{ij}\theta_{\lambda\mu}
(x_{\lambda\mu}A_{\lambda i}A_{\mu j}+x_{\lambda\mu}A_{\mu i}A_{\lambda j}) \\
 &= \theta_{ij}(\symm{A}{\hat X})_{\,ij}
\end{align*}
with the {\sl theta symbol\/} for pairs of single indices
$$ \theta_{ij}= 
\begin{cases}
1, & \text{if } i<j\cr  0, & \text{otherwise}
\end{cases}
$$
Note that the diagonal terms of $\hat X$ vanish anyway.
And $\symm{A}{\hat X}$ will be symmetric if $\hat X$ is.
Since the left-hand side has zero diagonal entries, we can remove the theta symbol and compensate by subtracting off
the diagonal, call it $D^+$. Taking traces yields \#1. And \#2 follows similarly.
\end{proof}

\begin{remark}
Observe that $D^+$ and $D^-$ may be explicitly given by
\begin{align*}
                    D^+_{\,ii}  &= 2\,x_{\lambda\mu}A_{\lambda i}A_{\mu i} \\
                    D^-_{\,ii}  &= 2\,x_{\lambda\mu}A_{i\lambda }A_{i\mu} 
\end{align*}
where for $D^+$ the $A$ elements are taken within a given column, while for $D^-$, the $A$ elements are in a given row.
\end{remark}

Connections between zeons and nonnegativity that we develop here give some indication that their natural
place is indeed in probability theory and quantum probability theory. \bigskip

Parallelling Proposition \ref{tpropXX}, we see the role of nonnegativity appearing.
\begin{proposition}\label{propXX} Let $X$ and $A$ be nonnegative. Then 
\begin{align*}
  \hat X = \symm{A}{\hat X}  &\Rightarrow  X\spow{A}{2} = X\\
  \hat X = \sym{A}{\hat X}  &\Rightarrow  \spow{A}{2}\dg  X. = \dg  X.\\
\end{align*}
\end{proposition}

\begin{proof} We have
$$D^+= \symm{A}{\hat X}-\Mat(X\spow{A}{2})$$
If $\hat X = \symm{A}{\hat X}$, then since $\hat X$ has vanishing trace, $\tr \symm{A}{\hat X}=0$. So
$\tr D^+=0$, hence $D^+=0$, and $\hat X = \symm{A}{\hat X}=\Mat(X\spow{A}{2})$. 
The second implication follows similarly. 
\end{proof}

\subsection{Trace Identities}

We continue with trace identities for zeons. The proofs follow from the above relations much as they do in the tensor case.

\begin{proposition}\label{propTR} We have \bigskip

1. $\displaystyle X\spow{A}{2}\dg  u. = \half \,\tr (\hat X A(J-I)A^*)$ .\bigskip

2. $\displaystyle u\spow{A}{2}\dg  X. = \half \,\tr (\hat X A^*(J-I)A)$ .\bigskip

3. If $A$ is stochastic, then \ $\displaystyle X\spow{A}{2}\dg  u. = \half \,\tr (\hat X (J-AA^*))$.
\end{proposition}

Using  equation (\ref{eqtrace1}) directly for $X$, we have
\begin{align}
X(I-\spow{A}{2})\dg  u.  &=  \half \,\tr (\hat X (J-AJA^*+AA^*)) \label{eqbp1}\\
u(I-\spow{A}{2})\dg  X.  &=  \half \,\tr (\hat X (J-A^*JA+A^*A))\label{eqbp2}
\end{align}

For zeons a new feature appears. 
For stochastic $A$, equation (\ref{eqbp1}) yields
\begin{lemma} [\rm ``integration-by-parts for zeons"]
\begin{equation} \label{eqparts}
X(I-\spow{A}{2})\dg  u. = \half \,\tr A^*\hat XA
\end{equation}
\end{lemma}

\subsection{Zeon hierarchy. $M$ and $N$ operators}

Consider the semigroup generated by the matrices $\spow{C_i}{\l}$, corresponding to the action on $\l$-sets of vertices.
The map $w\to \spow{w}{\l}$ is a homomorphism of matrix semigroups. 
We now have our main working tool in this context.
\begin{proposition} \label{prop:sup}
For $1\le \l\le n$, define
$$A_\l=\frac{1}{d} \spow{C_1}{\l}+\cdots +\frac{1}{d} \spow{C_d}{\l}\ .$$
Then the Ces\`aro limit of the powers $A_\l^{\,m}$ exists and equals 
$$\Omega_\l=\avg \spow{K}{\l}.$$
the average taken over the kernel $\K$ of the semigroup $\SG$ generated by $\{C_1,\ldots,C_d\}$.
\end{proposition}

With $r$ the rank of the kernel, we see that $\Omega_\l$ vanishes for $\l>r$.  \bigskip

Here we have $A_\l$ in general \textit{substochastic}, i.e., it may have zero rows or rows that do not sum to one.
As noted in the Appendix \ref{app:b}, the Abel limits exist for the $A_\l$. They agree with the Ces\`aro limits.

\subsection{$M$ and $N$ operators via zeons}

Here we will show how the operators $M$ and $N$ appear via zeons. The main tool are the trace identities.
Note that we automatically get $N_0$ while the diagonal of $M$ is cancelled out in the inner product with $\hat X$.

\begin{proposition}\label{propMN} For any $X$,
$$X\Omega_2\dg  u.=X\dg  u_2.=\avg X,u_2.=\half\,\tr \hat X\hat u_2 = \half\,\tr \hat XN_0$$
where $N_0=\avg J-KK^*.$, taken over the kernel.\bigskip

And
$$u\Omega_2\dg  X.=\pi_2\dg  X.=\avg \pi_2,X.=\half\,\tr \hat \pi_2\hat X = \half\,\tr \hat XM$$
where $M=\avg K^*JK.$, taken over the kernel.
\begin{proof} We have $\Omega_2=\avg \spow{K}{2}.$. Using Proposition \ref{propTR} for stochastic matrices, 
$$X\Omega_2\dg  u.=\avg X\spow{K}{2}{\dg  u.}.=\half\, \tr \avg \hat X(J-KK^*).= \half\,\tr \hat XN_0$$
Similarly,
$$u\Omega_2\dg  X.=\avg u\spow{K}{2}{\dg  X.}.=\half\, \tr \avg \hat XK^*(J-I)K).=
\half\, \tr \avg \hat XK^*JK. $$ 
where, since $K^*K$ is diagonal, $\tr \hat X K^*K=0$ as the diagonal of $\hat X$ vanishes.\hfill\qedhere
\end{proof}
\end{proposition}

From these relations we see that 
$$u_2^\dagger=\Omega_2\ud \quad \Rightarrow \quad N_0=\avg J-KK^*.=\Mat(u_2)$$
and
$$\pi_2=u\Omega_2 \quad \Rightarrow \quad M_0=\avg K^*JK.-\tau\omega=\Mat(\pi_2)$$

Cf., Proposition \ref{prop:level2}. Here $M$ and $N$ only require applying $\Omega_2$ to $u$.

\subsection{The spaces $\M_0$ and $\N_0$}

Now we have $$\displaystyle A_2=\half\,(\spow{R}{2}+\spow{B}{2})=\spow{A}{2}+\spow{\Delta}{2} \ .$$

Define the mappings $F,G\colon {\rm Sym}(\V)\to{\rm Sym}(\V)$ by
$$ F(Y)=\sym{A}{Y}+\sym{\Delta}{Y} \qquad\text{ and }\qquad G(Y)=\symm{A}{Y}+\symm{\Delta}{Y}\ .$$
consistent with
$$ F(Y)=\half\,(\sym{R}{Y}+\sym{B}{Y}) \qquad\text{ and }\qquad G(Y)=\half\,(\symm{R}{Y}+\symm{B}{Y})\ .\vstrut$$
First, some observations
\begin{proposition}  
For any vector $v$, \rm $G(\diag{v}) = \diag{vA}$ .
\end{proposition}
\begin{proof} Let $Y=G(\diag{v})$. Then
$2\,Y_{ij}=R_{\lambda i}v_{\lambda}R_{\lambda j}+B_{\lambda i}v_{\lambda}B_{\lambda j}$.
A term like $R_{li}$ means that $R$ maps $l$ to $i$. Since $R$ is a function, $l$ can only map to both $i$ and $j$ if $i=j$.
So $Y$ is diagonal. If $i=j$, then we get, using $R_{ij}^2=R_{ij}$, $B_{ij}^2=B_{ij}$,
$$2\,Y_{ii}=R_{\lambda i}v_{\lambda}+B_{\lambda i}v_{\lambda}=2\,v_\lambda A_{\lambda i}$$
as required.
\end{proof}

For example, taking $v=u$, we see that $$G(I)=I \ \text{if and only if}\  A \ \text{is doubly stochastic.}$$

First, observe
\begin{proposition} 
$F(J)=J$ and $G(\omega)=\omega$. That is, $J\in\N_+$ and $\omega\in \M_+$.
\end{proposition}
\begin{proof}
We have $F(J)=\sym{A}{J}+\sym{\Delta}{J}$. With $AJ=J=JA^*$ and $\Delta J=0$, this reduces to $J$.
Second, from the previous Proposition, we have $G(\omega)=G(\diag{\pi}))=\diag{\pi A}=\diag{\pi}=\omega$.
\end{proof}

Now recall Propositions \ref{prop:diagn} and \ref{prop:diagm}. We see that an element of $\M_+$ differs by a multiple of $\omega$
from an element of $\M_0$. For an element $N\in\N_+$ with entries bounded by 1, with diagonal entries all equal to one, we see that
$J-N$ gives a corresponding element in $\N_0$. We have seen the special operators $M$, $N$, $M_0$ and $N_0$ as interesting examples. \bigskip

Note that equation (\ref{eqparts}) applied to $R$ and $B$ and then averaged yields
\begin{equation} \label{eqPTS}
 X(I-A_2)\dg  u. =\half\,\tr G(\hat X)
\end{equation}

The next theorem shows the general relation between
fixed points of $A_2$ and fixed points of $F$ and $G$. Note that we are not restricting {\it a priori\/} to nonnegative solutions except
in case 2 of the theorem. This allows us to associate trace-zero $M$ and $N$ operators with averages over zeon powers of the kernel elements.

\begin{theorem} \hfill\bigskip

1. $F(\hat X)=\hat X$ if and only if $A_2\dg  X.=\dg  X.$.\medskip

2. For nonnegative solutions $X$, $G(\hat X)=\hat X$ if and only if $XA_2=X$.
\begin{proof}
First, from Proposition \ref{propBR}, applying the Basic Relations for $R$ and $B$ and averaging, we get
\begin{align}
\Mat(XA_2)  &= G(\hat X)-\half\,(D_R^++D_B^+) \label{eqG}\\
\Mat(A_2\dg  X.)  &=  F(\hat X)-\half\,(D_R^-+D_B^-)
\end{align}
where the subscripts on the $D$'s indicate the diagonals corresponding to $R$ and $B$ respectively.
Recalling the remark following Proposition \ref{propBR}, note that for $D_R^-$, say, we have terms like
$R_{i\lambda}x_{\lambda\mu}R_{i\mu}$. Since $R$ is a function, we must have $\lambda=\mu$, but then $x_{\lambda\mu}=0$.
In other words, $D_R^-$ and $D_B^-$ vanish so that, in fact,  $\Mat(A_2\dg  X.)  =  F(\hat X)$. And \#1 follows immediately.\bigskip

For \#2,  if $G(\hat X)=\hat X$, then the trace of the diagonal terms in equation (\ref{eqG}) vanish, so under the assumption of
nonnegativity, they must vanish. And replacing $G(\hat X)$ by $\hat X$, we have $XA_2=X$.\bigskip

Assume $XA_2=X$. Taking traces in equation \eqref{eqG}, the trace of the left-hand side vanishes and the trace of the diagonal
terms equals $\tr G(\hat X)$.  Now, use the integration-by-parts formula for $A_2$, equation \eqref{eqPTS}. The left-hand
side vanishes, hence $\tr G(\hat X)=0$, implying the vanishing of the terms $D_R^+$ and $D_B^+$ as well. Then equation
\eqref{eqG} reads $\Mat(X)=G(\hat X)$ as required. \hfill\qedhere
\end{proof}
\end{theorem}

\subsection{Quantum observables via zeons in summary}
We have $\Omega_2$ as the Ces\`aro limit of the powers $A_2^m$.
Then with $u$ here denoting the all-ones vector of dimension ${\binom{n}{2}}$, we have
$$u_2^\dagger=\Omega_2\dg  u.  \qquad\text{and}\qquad \pi_2=u\Omega_2$$

1. $N_0=\avg J-KK^*.=\Mat(u_2)$ satisfies $\sym{A}{N_0}+\sym{\Delta}{N_0}=N_0$. It is a nonnegative solution with $\tr N_0=0$.
$N=\avg KK^*.$ and $J=\dg  u.u$ are also nonnegative solutions, with trace equal to $n$.\bigskip

2. $M_0=M-\tau\omega$. From Proposition \ref{prop:level2}, we have $\avg K^*K.=n\omega$.
This allows us to express 
$$M_0=\avg K^*JK.-\tau/n \avg K^*K. \ .$$
$M_0=\Mat(\pi_2)$ satisfies $\symm{A}{M_0}+\symm{\Delta}{M_0}=M_0$. It is a nonnegative solution with $\tr M_0=0$.
$M=\avg K^*JK.$ and $\omega=\diag{\pi}$ are also nonnegative solutions, with trace equal to $\tau$ and $1$ respectively.\bigskip

3. We have $\tilde M=\avg K^*\Omega K.= nr^{-2}\avg {\dg \utilde\rho.}\utilde\rho.$.The diagonal of $\utilde\rho^\dagger \utilde\rho$
is $\rho$ so the diagonal of $\tilde M$ is $r\omega$. Thus, $\tilde M_0=\tilde M-r\omega$ has zero diagonal. 
Hence $\tilde M_0$ and $M_0$ are nonnegative, symmetric, with vanishing trace, elements of $\M_0$, corresponding to solutions to $XA_2=X$.
Notice that $\tr\Omega_2=1$ would imply that the space of such solutions equals 1 so in that case $\tilde M_0$ and $M_0$ are proportional. Thus,
$M_0$ could be computed using knowledge of the range classes.

\subsection{Levels in the zeon hierarchy}
At level 1, we have left and right eigenvectors of $A$: $\pi_1=\pi$ and $u_1^+=u^+$, respectively.
All of the components of $u$ are equal to 1. And $\Omega=\Omega_1=u^+\pi$.\bigskip

At every level $\ell$, $1\le \ell\le r$, we have an Abel limit 
$$\Omega_\ell=\lim_{s\uparrow1} (1-s)(I-sA_\ell)^{-1}$$
When $\pi_\ell$ and $u_\ell$ are unique, we have $\Omega_\ell=u_\ell^+\pi_\ell$.\bigskip

At level $r$, as at level 1, we have $\Omega_r=u_r^+\pi_r$.
\subsubsection{Interpretation}

Start at level $r$. Then the nonzero entries of $\pi_r$ are in the columns corresponding
to the ranges of the kernel, $\K$. The nonzero entries of $u_r^+$ are in rows corresponding
to cross-sections of the partition classes of $\K$. I.e., for each partition class, find all 
possible cross-sections weighted by one, take the total,  and you get a lower bound on the corresponding entries
of $u_r^+$. In particular, for a right group, all cross-sections appear with
the same weight, which can be scaled to 1.\bigskip

At each level, if unique, $\pi_\ell$ and $u_\ell$ give you information about the recurrent class and corresponding
cross-sections.\bigskip

A sequence of pairs $(\pi_\ell,u_\ell)$ can be constructed inductively from 
$(\pi_r,u_r)$ as consecutive marginals. Thus, for $\ell<r$, let
$$ \pi_\ell(\rI)=\sum_{\rJ\supset \rI}\pi_{\ell+1}(\rJ) $$
Note that in the sum each $\rJ$ differs from $\rI$ by a single vertex. Starting from level $r$, we 
inductively move down the hierarchy and determine a fixed point of $A_\ell$, $\pi_\ell$,  for each level $\ell$. \bigskip

For $u$'s, proceed as follows:
$$ u_\ell(\rI)=\sum_i p_i \,u_{\ell+1}(\rI\cup \{i\}) $$
where the weights $p_i$ are the components of the invariant measure $\pi$.\bigskip

With this construction, one recovers at level one, a multiple of $\pi_1=\pi$ and a
constant vector $u_1$, a multiple of $u$. \bigskip

Note that the rank $r$ is the maximum level $\ell$ such that $\Omega_\ell\ne0$.\bigskip

For our closing observation, we remark that 
by introducing an absorbing state at each level, $R_\ell$, $B_\ell$, and $A_\ell$, can be extended 
to stochastic matrices as at level one. \bigskip

\begin{remark} For further information as to the structure of the zeon hierarchy see \cite{BF}.
\end{remark}

\section{Conclusion}

At this point the interpretation of the eigenvalues and eigenvectors of the operators $M$ and $N$ is unavailable. It would
be quite interesting to discover their meaning. 
For further work,  it would be nice to have  similar constructions for at least some  classes of infinite semigroups, e.g. compact semigroups.
Another avenue to explore would be to develop such a theory for semigroups of operators on Hilbert space.

\appendix
\section{Semigroups and kernels}\label{sec:sgrpskrnls}
A finite semigroup, in general, is a finite set, $\SG$, which is closed under a binary associative operation. 
For any subset $T$ of $\SG$, we will write $\E(T)$ to refer to the set of idempotents in $T$. 
\begin{theorem}\cite[Th. 1.1]{HM}	\label{thm:kernel} 
Let $\SG$ be a finite semigroup. Then $\SG$ contains a minimal ideal $\K$
called the kernel which is a disjoint union of isomorphic groups. 
In fact, $\mathcal{K}$ is isomorphic to $\ms.X. \times \ms.G.  \times \ms.Y.$ where, given $e\in \E(\SG)$, then
$e\mathcal{K}e$ is a group and
\[
\ms.X. =\E( \K e)\,,\qquad  \ms.G. = e\K e\,, \qquad  \ms.Y. = \E(e\K)
\]
and if $\left( x_1,g_1,y_1\right) $, $\left( x_2,g_2,y_2\right)  \in \ms.X. \times \ms.G. \times \ms.Y.$
then the multiplication rule has the form
\[
(x_1, g_1, y_1) (x_2, g_2, y_2) = (x_1, g_1\phi( y_1,x_2) g_2, y_2)
\]
where $\phi\colon \ms.Y.\times\ms.X.\to\ms.G.$ is the \textit{sandwich function}.
\end{theorem}
The product structure $\ms.X. \times \ms.G. \times \ms.Y.$ is called a Rees product and any semigroup that has 
a Rees product is \textit{completely simple}. 
The kernel of a finite semigroup is always completely simple.

\subsection{Kernel of a matrix semigroup}  \label{ssec:kms}

An extremely useful characterization of the kernel for a semigroup of matrices is known.
\begin{theorem}	\cite[Props. 1.8, 1.9]{HM}	\label{thm:minrank}  
Let $\SG$ be a finite semigroup of matrices. Then the kernel, $\K$, of $\SG$ is the set of matrices with minimal rank.
\end{theorem}
Suppose $\SG=\SG((C_1,C_2,\ldots,C_d))$ is a semigroup generated by binary stochastic matrices $C_i$.
Then $\SG$ is a finite semigroup with kernel $\mathcal{K} = \ms.X. \times \ms.G. \times \ms.Y.$, the Rees product structure of
Theorem \ref{thm:kernel}. Let $k$, $k'$ be elements of $\mathcal{K}$. 
Then with respect to the Rees product structure $k=\left( k_1,k_2,k_3\right) $ and $k'=\left( k'_1,k'_2 ,k'_3\right) $. 
We form the ideals $\mathcal{K}k$, $k'\mathcal{K}$, and their intersection $k'\mathcal{K}k$: \bigskip

\begin{enumerate}
\item $\mathcal{K}k = \ms.X. \times \ms.G. \times \left\{k_3\right\}$ is a minimal left ideal in $\mathcal{K}$ whose elements all have the same range, or nonzero columns as $k$. 
We call this type of semigroup a \textit{left group}.  \bigskip

\item $k'\mathcal{K} = \left\{k'_1\right\} \times \ms.G. \times \ms.Y.$ is a minimal right ideal in $\mathcal{K}$ whose elements all have the same partition of the vertices as $k'$. 
A block $\B_j$ in the partition can be assigned to each nonzero column of $k'$ by 
\[
\B_j =\{ i : k'_{i,j}=1\}
\]
A semigroup with the structure $k'\mathcal{K}$ is a \textit{right group}.  \bigskip

\item $k'\mathcal{K}k$, the intersection of $k'\mathcal{K}$ and $\mathcal{K}k$, is a maximal group in $\mathcal{K}$ 
(an $H$-class in the language of semigroups). It is best thought of as the set of functions specified by the partition of $k'$ and the range of $k$,
acting as a group of permutations on the range of $k$.
The idempotent of $k'\mathcal{K}k$ is the function which is the identity when restricted to the range of $k$.
\end{enumerate}

\section{Abel limits}\label{app:b}
We are  working with stochastic and substochastic matrices. We give the details here on the
existence of Abel limits. Let $P$ be a substochastic matrix. The entries $p_{ij}$ satisfy
$0\le p_{ij}\le1$ with $\sum_jp_{ij}\le1$. Inductively, $P^n$ is substochastic for every integer $n>0$, with
$P^n$ stochastic if $P$ is.

\begin{proposition}     
The Abel limit
$$\Omega=\lim_{s\uparrow1}\, (1-s)(I-sP)^{-1}$$
exists and satisfies
$$ \Omega=\Omega^2=P\Omega=\Omega P$$
\end{proposition}
\begin{proof}     
We take $0<s<1$ throughout.\bigskip

For each $i,j$, the matrix elements $\avg e_i,P^ne_j.$ are uniformly bounded by 1. Denote
$$Q(s)=(1-s)(I-sP)^{-1}$$
So we have
$$\avg e_i,Q(s)e_j.=(1-s)\sum_{n=0}^\infty s^n\avg e_i,P^ne_j.\le (1-s)\sum_{n=0}^\infty s^n=1$$
so the matrix elements of $Q(s)$ are bounded by 1 uniformly in $s$ as well. Now take any sequence $\{s_j\}$,
$s_j \uparrow 1$. Take a further subsequence $\{s_{j11}\}$ along which the $11$ matrix elements converge.
Continuing with a diagonalization procedure going successively through the matrix elements, we
have a subsequence $\{s'\}$ along which all of the matrix elements converge, i.e., along which
$Q(s')$ converges. Call the limit $\Omega$.\bigskip

Writing $(I-sP)^{-1}-I=sP(I-sP)^{-1}$, multiplying by $1-s$ and letting $s \uparrow 1$ along $s'$ yields
$$\Omega=P\Omega=\Omega P$$
writing the $P$ on the other side for this last equality. Now, $s\Omega=sP\Omega$ 
implies $(1-s)\Omega=(I-sP)\Omega$. Similarly,  $(1-s)\Omega=\Omega(I-sP)$, so 
\begin{equation}\label{eq:abel} Q(s)\Omega=\Omega Q(s)=\Omega \end{equation}
Taking limits along $s'$ yields $\Omega^2=\Omega$.\bigskip

For the limit to exist, we check that $\Omega$ is the only limit point of $Q(s)$. From \eqref{eq:abel},
if $\Omega_1$ is any limit point of $Q(s)$, $\Omega_1\Omega=\Omega\Omega_1=\Omega$. 
Interchanging r\^oles of $\Omega_1$ and $\Omega$ yields 
$\Omega_1\Omega=\Omega\Omega_1=\Omega_1$ as well, i.e., $\Omega=\Omega_1$.\hfill\qedhere
\end{proof}

If $P$ is stochastic, it has $u^+$ as a nontrivial fixed point. In general we have

\begin{corollary}  
Let $\Omega=\lim\limits_{s\uparrow1}\, (1-s)(I-sP)^{-1}$ be the Abel limit of the powers of $P$. \bigskip

$P$ has a nontrivial fixed point if and only if $\Omega\ne0$.\bigskip
\end{corollary}
\begin{proof}  
If $\Omega\ne0$, then $P\Omega=\Omega$ shows that any nonzero column of $\Omega$ is a nontrivial fixed point.
On the other hand, if $Pv=v\ne0$, then, as in the above proof, $Q(s)v=v$ and hence $\Omega v=v$ shows that $\Omega\ne0$.\hfill\qedhere
\end{proof}
Note that since $\Omega$ is a projection, its rank is the dimension of the space of fixed points.  \bigskip

{\bf Acknowledgment.}\ 
The author is indebted to G.\,Budzban for introducing him to this subject. \bigskip

This article is motivated by the QP32 Conference in Levico, Italy, May-June, 2011. The author
wishes to thank the organizers for a truly enjoyable meeting.

\end{document}